\documentclass[a4paper,12pt]{article}
\usepackage{color}
\usepackage[english]{babel}
\usepackage{amsfonts}
\usepackage{amsmath}
\usepackage{amsthm}
\usepackage{amssymb}
\usepackage{bbm}
\usepackage{mathrsfs}
\usepackage{esint}
\usepackage{faktor}
\usepackage[utf8]{inputenc}
\usepackage{afterpage}
\usepackage{graphicx}

\newcommand{\N}{\mathbb{N}}

\newcommand{\R}{\mathbb{R}}

\renewcommand{\epsilon}{\varepsilon}
\renewcommand{\phi}{\varphi}

\newtheorem{lemma}{Lemma}[section]
\newtheorem{thm}[lemma]{Theorem}

\theoremstyle{definition}

\newtheorem{rmk}[lemma]{Remark}
\newtheorem{ex}[lemma]{Example}

\numberwithin{equation}{section}
\DeclareMathOperator*{\esssup}{ess \, sup}
\DeclareMathOperator*{\essinf}{ess \, inf}
\DeclareMathOperator*{\supp}{supp}

\begin{document}
\title{\textbf{Singular quasilinear convective \\ elliptic systems in $\R^N$}}
\author{
\bf Umberto Guarnotta, Salvatore Angelo Marano\thanks{Corresponding Author.}\\
\small{Dipartimento di Matematica e Informatica, Universit\`a di Catania,}\\
\small{Viale A. Doria 6, 95125 Catania, Italy}\\
\small{\it E-mail: umberto.guarnotta@phd.unict.it, marano@dmi.unict.it}\\
\mbox{}\\
\bf Abdelkrim Moussaoui\\
\small{LMA, Faculty of Exact Sciences, A. Mira Bejaia University,}\\
\small{06000 Bejaia, Algeria}\\
\small{\it E-mail: abdelkrim.moussaoui@univ-bejaia.dz}
}
\date{}
\maketitle
\begin{abstract}
The existence of a positive entire weak solution to a singular quasi-linear elliptic system with convection terms is established, chiefly through perturbation techniques, fixed point arguments, and a priori estimates. Some regularity results are then employed to show that the obtained solution is actually strong.
\end{abstract}
\vspace{2ex}
\vspace{2ex}


\section{Introduction and main result}
In this paper, we deal with the problem
\begin{equation}
\label{prob}
\tag{$ {\rm P} $}
\left\{
\begin{alignedat}{2}
-\Delta_p u &= f(x,u,v,\nabla u, \nabla v) \; \; &&\mbox{in} \; \; \R^N, \\
-\Delta_q v &= g(x,u,v,\nabla u, \nabla v) \; \; &&\mbox{in} \; \; \R^N, \\
u,v &> 0 \; \; &&\mbox{in} \; \; \R^N,
\end{alignedat}
\right.
\end{equation}
where $ N \geq 3 $, $ 2-\frac{1}{N} < p,q < N $, $ \Delta_r z := {\rm div}(|\nabla z|^{r-2}\nabla z) $ denotes the $ r $-Laplacian of $ z $ for $ 1 < r < +\infty $, while $ f,g: \R^N \times (0,+\infty)^2 \times \R^{2N} \to (0,+\infty) $ are Carathéodory functions satisfying assumptions $ {\rm H_1} $--$ {\rm H_3} $ below.

Problem \eqref{prob} exhibits three interesting features:
\begin{itemize}
\item The reaction terms $ f $ and $ g $ can be singular at zero.
\item $ f $, $ g $ depend on the gradient of solutions.
\item Equations are set in the whole space $ \R^N $.
\end{itemize}
However, they give rise to some nontrivial difficulties, such as the loss of variational structure and the lack of compactness for Sobolev embedding. This work continues the study started in \cite{MMM}, whose setting was $\R^N$ and convective terms did not appear, along the very recent papers \cite{LMZ,CLM,GM,GMM}, which address analogous questions, but concerning a bounded domain. 

Primarily, we need an appropriate functional framework where to treat the problem, mainly because the integrability properties of solutions and their gradients may differ at infinity, as the example in \cite[p. 80]{G} shows. Accordingly, one is led to employ the so-called Beppo Levi (or homogeneous Sobolev) spaces $\mathcal{D}^{1,r}_0(\R^N)$, systematically studied for the first time by Deny and Lions \cite{DL}. The monographs \cite{G,LL,SS} provide an exhaustive introduction on the topic. 

Let $X:=\mathcal{D}^{1,p}_0(\R^N)\times\mathcal{D}^{1,q}_0(\R^N)$ and let $r'$  denote the conjugate exponent of $r>1$. A pair $(u,v)\in X$ such that $u,v>0$ a.e. in $\R^N$ is called:
\begin{itemize}
\item[1)] \emph{distributional solution} to \eqref{prob} if for every $(\phi_1,\phi_2)\in C^{\infty}_0(\R^N)^2$ one has
\begin{equation}\label{intbyparts}
\begin{split}
\int_{\R^N}|\nabla u|^{p-2}\nabla u\nabla\phi_1 {\rm d}x &= \int_{\R^N} f(\cdot,u,v,\nabla u,\nabla v) \phi_1 {\rm d}x, \\
\int_{\R^N}|\nabla v|^{q-2}\nabla v\nabla\phi_2 {\rm d}x &= \int_{\R^N} g(\cdot,u,v,\nabla u,\nabla v) \phi_2 {\rm d}x;
\end{split}
\end{equation}
\item[2)] \emph{(weak) solution} of \eqref{prob} when \eqref{intbyparts} holds for all $(\phi_1,\phi_2)\in X$;
\item[3)] \emph{`strong' solution} to \eqref{prob} if $|\nabla u|^{p-2}\nabla u, |\nabla v|^{q-2}\nabla v \in W^{1,2}_{\rm loc}(\R^N)$ and the differential equations are satisfied a.e. in $\R^N$.
\end{itemize}
Obviously, both 2) and 3) force 1), whilst reverse implications turn out generally false; see also Remark \ref{equivalence}. Moreover, as observed at p. 48 of \cite{SS}, problems in unbounded domains may admit strong solutions that are not weak or vice-versa. So, the search for strong solutions appears of some interest in this context.

Roughly speaking, our technical approach proceeds as follows. We first solve an auxiliary problem \eqref{aux2},
$\epsilon>0$, obtained by shifting variables of reactions, which avoids singularities. To do this, nonlinear regularity theory, a priori estimates, Moser's iteration, trapping region, and fixed point arguments are employed. Unfortunately, bounds from above alone do not allow to get a solution of \eqref{prob}: treating singular terms additionally requires some estimates from below. Theorem 3.1 in \cite{DM} ensures that solutions to \eqref{aux2} turn out locally greater than a positive constant regardless of $\epsilon$. Thus, under hypotheses
${\rm H}_1$--${\rm H}_3$ below, we can construct a sequence $\{(u_\epsilon,v_\epsilon)\}\subseteq X$ such that $(u_\epsilon,v_\epsilon)$ solves \eqref{aux2} for all $\epsilon>0$ and whose weak limit as $\epsilon\to 0^+$ is a distributional solution to \eqref{prob}; cf. Lemma \ref{distrsol}. Next, a localization-regularization reasoning (see Lemma \ref{weaksol}) shows that 
$$(u,v)\;\mbox{distributional solution}\implies (u,v)\;\mbox{weak solution.}$$
Through a recent differentiability result \cite[Theorem 2.1]{CM} one then has
$$(u,v)\;\mbox{distributional solution}\implies (u,v)\;\mbox{strong solution;}$$
cf. Lemma \ref{strongsol}. Further, $(u,v)\in C^{1,\alpha}_{\rm loc}(\R^N)$ once condition ${\rm H}_3$ is slightly strengthened; see Remark \ref{regularity}. 

Singular elliptic problems in $\R^N$ have a long history, that traces back to \cite{KS,D,C,CK,LS} for semi-linear equations. More recent results, involving also systems, can be found in \cite{MKT,DS,MMM,SLARZ} and the references therein.

Henceforth, the assumptions below will be posited. If $1<r<N$ then, by definition, $r^*:=\frac{Nr}{N-r}$.
\begin{itemize}
\item[$ \underline{\rm H_1(f)} $] There exist $\alpha_1\in (-1,0]$, $\beta_1,\delta_1\in [0, q-1)$, $\gamma_1\in [0,p-1)$, $ m_1, \hat{m}_1 > 0 $, and $ a_1 \in L^{s_p}_{\rm loc}(\R^N) $, with $ s_p > p'N $, such that
\begin{equation*}
m_1 a_1(x) s_1^{\alpha_1} s_2^{\beta_1} \leq f(x,s_1,s_2,\xi_1,\xi_2) \leq \hat{m}_1 a_1(x) \left( s_1^{\alpha_1} s_2^{\beta_1} + |\xi_1|^{\gamma_1} + |\xi_2|^{\delta_1} \right)
\end{equation*}
in $\R^N \times (0,+\infty)^2 \times\R^{2N}$. Moreover, $\displaystyle{\essinf_{B_\rho} a_1 > 0}$ for all $\rho>0$.
\item[$ \underline{\rm H_1(g)} $] There exist $\beta_2\in (-1,0]$, $\alpha_2, \gamma_2\in [0,p-1)$, $\delta_2 \in [0,q-1)$, $ m_2, \hat{m}_2 > 0 $, and $ a_2 \in L^{s_q}_{\rm loc}(\R^N) $, with $ s_q > q'N $, such that
\begin{equation*}
m_2 a_2(x) s_1^{\alpha_2} s_2^{\beta_2} \leq g(x,s_1,s_2,\xi_1,\xi_2) \leq \hat{m}_2 a_2(x) \left( s_1^{\alpha_2} s_2^{\beta_2} + |\xi_1|^{\gamma_2} + |\xi_2|^{\delta_2} \right)
\end{equation*}
in $ \R^N \times (0,+\infty)^2 \times\R^{2N} $. Moreover, $\displaystyle{\essinf_{B_\rho} a_2 > 0}$ for all $\rho>0$.
\item[$ \underline{\rm H_1(a)} $] There exist $ \zeta_1,\zeta_2 \in (N,+\infty] $ such that $ a_i \in L^1(\R^N) \cap L^{\zeta_i}(\R^N) $, $ i=1,2 $, where
\begin{equation*}
\frac{1}{\zeta_1} < 1 - \frac{p}{p^*} - \theta_1, \quad \frac{1}{\zeta_2} < 1 - \frac{q}{q^*} - \theta_2,
\end{equation*}
with
\begin{equation*}
\begin{split}
\theta_1 :=\max\left\{\frac{\beta_1}{q^*},\frac{\gamma_1}{p},\frac{\delta_1}{q}\right\} < 1 - \frac{p}{p^*}, \quad\theta_2 :=\max\left\{\frac{\alpha_2}{p^*},\frac{\gamma_2}{p},\frac{\delta_2}{q}\right\}< 1-
\frac{q}{q^*}.
\end{split}
\end{equation*}
\item[$ \underline{\rm H_2} $] If $ \eta_1 := \max\{\beta_1,\delta_1\} $ and $ \eta_2 := \max\{\alpha_2,\gamma_2\} $ then
\begin{equation*}
\eta_1 \eta_2 < (p-1-\gamma_1)(q-1-\delta_2).
\end{equation*}
\item[$ \underline{\rm H_3} $] One has
\begin{equation*}
\begin{split}
\frac{1}{s_p} + \max \left\{ \frac{\gamma_1}{p},\frac{\delta_1}{q} \right\} \leq \frac{1}{2}, \qquad
\frac{1}{s_q} + \max \left\{ \frac{\gamma_2}{p},\frac{\delta_2}{q} \right\} \leq \frac{1}{2}.
\end{split}
\end{equation*}
%
\end{itemize}

\begin{ex}
${\rm H_1(a)}$ is fulfilled once $a_1,a_2 \in L^1(\R^N) \cap L^\infty(\R^N) $ and
\begin{equation*}
\max \left\{ \frac{\beta_1}{q^*},\frac{\gamma_1}{p},\frac{\delta_1}{q} \right\}< 1-\frac{p}{p^*}, \quad
\max \left\{ \frac{\alpha_2}{p^*},\frac{\gamma_2}{p},\frac{\delta_2}{q} \right\}< 1-\frac{q}{q^*}.
\end{equation*}
In fact, it suffices to choose $ \zeta_1:= \zeta_2:= +\infty $.
\end{ex}

\begin{rmk}
By interpolation (see, e.g., \cite[Proposition 2.1]{MMM}), condition $ {\rm H_1(a)} $ entails $ a_i \in L^{\sigma_{i,j}}(\R^N) $, $ i=1,2 $, where:
\begin{itemize}
\item[(i)] $ \sigma_{1,j} := \frac{1}{1-t_j} $, $ j=1,2,3,4 $, with
\begin{equation*}
t_1 = \frac{\alpha_1+1}{p^*} + \frac{\beta_1}{q^*}, \quad t_2 = \frac{1}{p^*} + \frac{\beta_1}{q^*}, \quad t_3 = \frac{1}{p^*} + \frac{\gamma_1}{p}, \quad t_4 = \frac{1}{p^*} + \frac{\delta_1}{q};
\end{equation*}
\item[(ii)] $ \sigma_{2,j} := \frac{1}{1-t_j} $, $ j=1,2,3,4 $, with
\begin{equation*}
t_1 = \frac{\beta_2+1}{q^*} + \frac{\alpha_2}{p^*}, \quad t_2 = \frac{1}{q^*} + \frac{\alpha_2}{p^*}, \quad t_3 = \frac{1}{q^*} + \frac{\gamma_2}{p}, \quad t_4 = \frac{1}{q^*} + \frac{\delta_2}{q}.
\end{equation*}
\end{itemize}
\end{rmk}
The aim of this paper is to prove the following
\begin{thm}\label{main}
Under hypotheses ${\rm H_1}$--${\rm H_3}$, problem \eqref{prob} admits a weak and strong solution $(u,v) \in X$.
\end{thm}
\section{Preliminaries}
Let $Z$ be a Hausdorff topological space and let $T:Z\to Z$ be continuous. Following \cite[p. 2]{GD}, the operator $T$ is called compact when $\overline{T(Z)}$ turns out a compact subset of $Z$. If $Z$ is a normed space, $ \{z_n\}\subseteq Z $, and $z\in Z$ then $z_n\to z$ in $Z$ means that the sequence $\{z_n\}$ strongly converges to $z$, while $ z_n \rightharpoonup z $ stands for weak convergence. As usual, $Z^*$ denotes the topological dual of $Z$ and $ Z^2:= Z \times Z $.

Hereafter, $ N \geq 3 $ is a fixed integer, $ B(x,\rho)$ indicates the open ball in $ \R^N $ of radius $\rho>0$ centered at $x\in\R^N$ and $B_\rho:=B(0,\rho)$, while $|E|$ stands for the Lebesgue measure of $E$.

Let $ Z:= Z(\Omega) $ be a real-valued function space on a nonempty measurable set $\Omega \subseteq \R^N$. If $z_1,z_2\in Z$ and $z_1(x)<z_2(x)$ a.e. in $\Omega$ then we simply write $z_1<z_2$. The meaning of $z_1\leq z_2$, etc. is analogous. Put
$$Z_+:=\{ z\in Z: z>0\}.$$
Given $ \{z_n\}\subseteq Z $ and $z\in Z$, the symbol  $ z_n\uparrow z$ signifies that $\{z_n\}$ is monotone increasing and $ z_n(x) \to z(x) $ for almost every $x\in \Omega $. Moreover,
$$ z^\pm := \max\{\pm z,0\},\quad \supp z := \overline{\{x \in\Omega: z(x) \neq 0\}}\, .$$
When $\Omega:=\R^N$ and for every nonempty compact subset $ K $ of $\R^N$ the restriction $z\lfloor_K$ belongs to $Z(K)$ then we write $z\in Z_{\rm loc}(\R^N)$. Similarly, a sequence $\{z_n\}\subseteq Z_{\rm loc}(\R^N)$ is called bounded in $ Z_{\rm loc}(\R^N) $ once the same holds for $ \{z_n\lfloor_K\} $ in $ Z(K) $, with any $K$ as above. 

Let $ 1 < r < N $ and let $ z:\R^N \to \R $ a measurable function. Throughout the paper, $r':=\frac{r}{r-1}$, $r^*:=\frac{Nr}{N-r}$,
\begin{equation*}
\|z\|_r := \left( \int_{\R^N} |z(x)|^r {\rm d}x \right)^{1/r},\quad\|z\|_\infty := \esssup_{x\in\R^N} |z(x)|\, .
\end{equation*}
We now recall the notion and some relevant properties of Beppo Levi's space $ \mathcal{D}^{1,r}_0(\R^N) $, addressing the reader to \cite[Chapter II]{G} for a complete treatment. Put
\begin{equation*}
\mathcal{D}^{1,r} := \left\{ z \in L^1_{\rm loc}(\R^N): |\nabla z| \in L^r(\R^N) \right\}
\end{equation*}
and denote by $ \mathcal{R} $ the equivalence relation that identifies two elements in $ \mathcal{D}^{1,r} $ whose difference is a constant. The quotient set $\mathcal{\dot D}^{1,r}$, endowed with the norm
$$\|z\|_{1,r} :=\left(\int_{\R^N}|\nabla z(x)|^r{\rm d}x\right)^{1/r},$$
turns out complete. Write $\mathcal{D}^{1,r}_0(\R^N)$ for the subspace of $ \mathcal{\dot D}^{1,r} $ defined as the closure of $ C^\infty_0(\R^N) $ under $ \|\cdot\|_{1,r} $, namely
$$\mathcal{D}^{1,r}_0(\R^N) := \overline{C^\infty_0(\R^N)}^{\|\cdot\|_{1,r}}.$$
$\mathcal{D}^{1,r}_0(\R^N)$, usually called Beppo Levi space,  is reflexive and continuously embeds in $L^{r^*}(\R^N)$, i.e., 
\begin{equation}\label{embedding}
\mathcal{D}^{1,r}_0(\R^N) \hookrightarrow L^{r^*}(\R^N).
\end{equation}
Consequently, if $ z \in \mathcal{D}^{1,r}_0(\R^N) $ then $z$ vanishes at infinity, meaning that the set $\{x \in \R^N:|z(x)| \geq \epsilon\} $ has finite measure for any $ \epsilon > 0 $. In fact, by Chebichev's inequality and \eqref{embedding}, one has
\begin{equation*}
|\{x \in \R^N: \, |z(x)| \geq \epsilon \}| \leq \epsilon^{-r^*} \|z\|_{r^*}^{r^*} \leq (c \epsilon^{-1} \|z\|_{1,r})^{r^*} < +\infty,
\end{equation*}
where $ c > 0 $ is the best constant related to \eqref{embedding}; see the seminal paper \cite{T}. 

Hereafter, $ c $, $ c_\epsilon $, and $ c_\epsilon(\cdot)$ will denote generic positive constants, which may change explicit value from line to line. Subscripts and/or arguments emphasize their dependence on a given variable.

To avoid cumbersome expressions, define 
$$ X:= \mathcal{D}^{1,p}_0(\R^N) \times \mathcal{D}^{1,q}_0(\R^N)\, ,\quad
 \|(u,v)\| := \|u\|_{1,p} + \|v\|_{1,q}\quad\forall\, (u,v)\in X .$$
$$\mathcal{C}^1_+ := X_+ \cap C^1_{\rm loc}(\R^N)^2, \quad \mathcal{C}^{1,\alpha}_+ := X_+ \cap C^{1,\alpha}_{\rm loc}(\R^N)^2.$$
$\mathcal{C}^1_+$ and $\mathcal{C}^{1,\alpha}_+$ will be endowed with the topology induced by that of $X$.

The following a priori estimate will play a basic role in the sequel.
\begin{lemma}\label{Mingione}
Let $2-\frac{1}{N} < r < +\infty $, let $ \zeta > N $, and let $ h \in L^1(\R^N) \cap L^\zeta(\R^N) $. If $ z \in \mathcal{D}^{1,r}_0(\R^N)\cap C^1_{{\rm loc}}(\R^N)$ is a weak solution to $-\Delta_r z = h(x)$ in $ \R^N$
then there exists $ c > 0 $, independent of $ z $, such that
\begin{equation*}
\|\nabla z\|_\infty^{r-1} \leq c \inf_{R>0} \left( R^{1-\frac{N}{\zeta}} \|h\|_\zeta + R^{-\frac{N}{r'}} \|\nabla z\|_r^{r-1} \right).
\end{equation*}
\end{lemma}
\begin{proof}
Pick any $x\in\R^N$ and $ R > 0 $. Via \cite[Theorem 1.1]{KM}, when $ r \geq 2 $, or \cite[Theorem 1.1]{DuMi}, if $2-\frac{1}{N} < r < 2 $, with $ \Omega:=B(x,R) $ and $ \mu:= h {\rm d}x $, as well as H\"older's inequality, we easily get
\begin{equation*}
\begin{split}
|\nabla z(x)|^{r-1} &\leq c\left[\int_0^R \left( \rho^{-N} \int_{B(x,\rho)} |h|\,{\rm d}y \right){\rm d}\rho
+\left( R^{-N} \int_{B(x,R)} |\nabla z|\,{\rm d}y \right)^{r-1} \right]\\
&\leq c\left(\|h\|_\zeta\int_0^R\rho^{-\frac{N}{\zeta}}{\rm d}\rho+R^{-\frac{r-1}{r}N}\|\nabla z\|_r^{r-1} \right)\\
&\leq c \left( R^{1-\frac{N}{\zeta}} \|h\|_\zeta + R^{-\frac{N}{r'}} \|\nabla z\|_r^{r-1} \right),
\end{split}
\end{equation*}
where $c>0$ does not depend on $z$, $h$, $x$, and $R$; see \cite[Remark 1.3]{DuMi}. Taking the infimum in $ R>0$ on the right and the supremum in $x \in \R^N$ on the left yields the conclusion.
\end{proof}
\section{The regularized system}
\subsection{`Freezing' the right-hand side}
Fix $ w:= (w_1,w_2) \in \mathcal{C}^1_+ $, $ \epsilon > 0 $ and define
\begin{equation*}
f_{w,\epsilon}:=f(\cdot,w_1+\epsilon,w_2,\nabla w),\quad
g_{w,\epsilon}:=g(\cdot,w_1,w_2+\epsilon,\nabla w),
\end{equation*}
where $ \nabla w := (\nabla w_1,\nabla w_2) $. We first focus on the auxiliary problem
\begin{equation}\label{aux}
\tag{$ {\rm P}^\epsilon_w $}
\left\{
\begin{alignedat}{2}
-\Delta_p u &= f_{w,\epsilon}(x) \; \; &\mbox{in} \; \; \R^N, \\
-\Delta_q v &= g_{w,\epsilon}(x)\; \; &\mbox{in} \; \; \R^N.
\end{alignedat}
\right.
\end{equation}
\begin{lemma}\label{auxsol}
If ${\rm H_1}$ holds then \eqref{aux} admits a unique solution $ (u,v) \in \mathcal{C}^{1,\alpha}_+ $.
\end{lemma}
\begin{proof}
Hypothesis $ {\rm H_1} $ and \eqref{embedding} guarantee that $(f_{w,\epsilon},g_{w,\epsilon})\in X^* $. Hence, by Minty-Browder's Theorem \cite[Theorem 5.16]{B}, problem \eqref{aux} possesses a unique solution $ (u,v) \in X $. Thanks to $ {\rm H_1}$ again one has
$$(f_{w,\epsilon},g_{w,\epsilon})\in L^{s_p}_{\rm loc}(\R^N) \times L^{s_q}_{\rm loc}(\R^N).$$
Thus, standard results from nonlinear regularity theory \cite[p. 830]{DB} entail $ (u,v) \in C^{1,\alpha}_{\rm loc}(\R^N)^2 $. Testing the first equation in \eqref{aux} with $ u^- $ we next obtain
\begin{equation*}
-\| \nabla u^-\|_p^p = \int_{\R^N} f_{w,\epsilon}u^- {\rm d}x \geq 0,
\end{equation*}
because $f$ is non-negative, which forces $ u \geq 0 $. Likewise, $ v \geq 0 $. The strong maximum principle \cite[Theorem 1.1.1]{PS} finally yields $ (u,v) \in X_+ $.
\end{proof}
\emph{Throughout this sub-section, $(u,v)$ will denote the solution to \eqref{aux} given by Lemma \ref{auxsol}.}
\begin{lemma}\label{auxunifest}
Let $ {\rm H_1} $ be satisfied. Then there exists $ L_\epsilon > 0 $ such that
\begin{equation*}
\begin{split}
\|\nabla u\|_p^{p-1} &\leq L_\epsilon (1 + \|\nabla w_1\|_p^{\gamma_1} + \|\nabla w_2\|_q^{\eta_1}), \\
\|\nabla v\|_q^{q-1} &\leq L_\epsilon (1 + \|\nabla w_1\|_p^{\eta_2} + \|\nabla w_2\|_q^{\delta_2}),
\end{split}
\end{equation*}
where $ \eta_1 := \max\{\beta_1,\delta_1\} $ and $\eta_2 := \max\{\alpha_2,\gamma_2\}$.
\end{lemma}
\begin{proof}
Test the first equation in \eqref{aux} with $ u $ and exploit $ {\rm H_1(f)} $, $ {\rm H_1(a)} $, besides \eqref{embedding}, to achieve
\begin{equation}\label{auxuest}
\begin{split}
\|\nabla u\|_p^p &= \int_{\R^N} f(\cdot,w_1+\epsilon,w_2,\nabla w_1,\nabla w_2) u {\rm d}x \\
&\leq \hat{m}_1 \int_{\R^N} a_1 [(w_1+\epsilon)^{\alpha_1} w_2^{\beta_1} + |\nabla w_1|^{\gamma_1} + |\nabla w_2|^{\delta_1}] u {\rm d}x \\
&\leq \hat{m}_1 \int_{\R^N} a_1 \max\{1,\epsilon^{\alpha_1}\} (w_2^{\beta_1} + |\nabla w_1|^{\gamma_1} + |\nabla w_2|^{\delta_1}) u {\rm d}x \\
&\leq c_\epsilon \|u\|_{p^*}(\|w_2\|_{q^*}^{\beta_1} + \|\nabla w_1\|_p^{\gamma_1} + \|\nabla w_2\|_q^{\delta_1}) \\
&\leq c_\epsilon \|\nabla u\|_p (\|\nabla w_2\|_q^{\beta_1} + \|\nabla w_1\|_p^{\gamma_1} + \|\nabla w_2\|_q^{\delta_1}) \\
&\leq L_\epsilon\|\nabla u\|_p (1 + \|\nabla w_1\|_p^{\gamma_1} + \|\nabla w_2\|_q^{\eta_1}), 
\end{split}
\end{equation}
because 
\begin{equation}\label{useful}
\Vert\nabla w_2\Vert_q^{\beta_1}+\Vert\nabla w_2\Vert_q^{\delta_1}\leq 2(1+\Vert\nabla w_2\Vert_q^{\eta_1}).\end{equation}
This shows the first inequality. The other is verified similarly.
\end{proof}
\begin{lemma}\label{auxmoser}
Under $ {\rm H_1} $, there exists $ M_\epsilon:=M_\epsilon(\|\nabla w_1\|_p, \|\nabla w_2\|_q) > 0 $ such that
\begin{equation*}
\max\{\|u\|_\infty,\|v\|_\infty\} \leq M_\epsilon(\|\nabla w_1\|_p, \|\nabla w_2\|_q).
\end{equation*}
\end{lemma}
\begin{proof}
The proof can be made by adapting those of Lemmas 3.2--3.3 in \cite{MMM}. So, we will briefly focus the key-points only. Fix any $ \xi_1 \in [1,\frac{p^*}{p}) $ such that
\begin{equation}\label{xione}
\frac{1}{\zeta_1} < 1 - \frac{1}{\xi_1} - \theta_1,
\end{equation}
where $ \zeta_1 $ and $ \theta_1 $ come from $ {\rm H_1(a)} $. Set $ u_K := \min\{u,K\} $,  $K>1$, and test \eqref{aux} with $ \phi := u_K^{kp+1} $, $ k \geq 0 $. Fatou's Lemma, H\"older's inequality joined to $ {\rm H_1(a)} $, Sobolev's embedding \eqref{embedding}, and \eqref{useful} produce
\begin{equation*}\label{mosercomp}
\begin{split}
\frac{kp+1}{(k+1)^p}\,\|u\|_{(k+1)p^*}^{(k+1)p} &\leq \frac{kp+1}{(k+1)^p}\,\liminf_{K \to +\infty}
\|u_K\|_{(k+1)p^*}^{(k+1)p}  \\
&\leq c \max\{1,\epsilon^{\alpha_1}\} \int_{\R^N} a_1 (w_2^{\beta_1}+|\nabla w_1|^{\gamma_1}+|\nabla w_2|^{\delta_1}) u^{kp+1} {\rm d}x \\
&\leq c_\epsilon (\|\nabla w_2\|_q^{\beta_1} + \|\nabla w_1\|_p^{\gamma_1} + \|\nabla w_2\|_q^{\delta_1}) \|u\|_{(kp+1)\xi_1}^{kp+1} \\
&\leq c_\epsilon (1 + \|\nabla w_1\|_p^{\gamma_1} + \|\nabla w_2\|_q^{\eta_1}) \|u\|_{(kp+1)\xi_1}^{kp+1};
\end{split}
\end{equation*} 
cf. \cite[pp. 1587--1588]{MMM}. Moreover, 
\begin{equation*}
(k+1)p^* > (kp+1)\xi_1 \quad\forall\, k \geq 0
\end{equation*}
as $\xi_1 <\frac{p^*}{p}$. Hence, Moser's iteration can start, and we obtain $\|u\|_\infty\leq M_\epsilon$,  where
$$M_\epsilon:=c_\epsilon (1 + \|\nabla w_1\|_p^{\gamma_1} + \|\nabla w_2\|_q^{\eta_1})^\tau$$
for some $\tau>0$; details can be read in \cite[pp. 1588--1590]{MMM}. A similar argument applies to $ v $.
\end{proof}
\begin{lemma}\label{auxgradest}
If $ {\rm H_1} $ holds and $ \max\{\|w_i\|_\infty,\|\nabla w_i\|_\infty\} < +\infty $, $ i=1,2 $, then 
\begin{equation*}
\begin{split}
\|\nabla u\|_\infty^{p-1} &\leq N_\epsilon(\|\nabla w_1\|_p,\|\nabla w_2\|_q,\|w_2\|_\infty) (1 + \|\nabla w_1\|_\infty^{\gamma_1} + \|\nabla w_2\|_\infty^{\delta_1}), \\
\|\nabla v\|_\infty^{q-1} &\leq N_\epsilon(\|\nabla w_1\|_p,\|\nabla w_2\|_q,\|w_1\|_\infty) (1 + \|\nabla w_1\|_\infty^{\gamma_2} + \|\nabla w_2\|_\infty^{\delta_2})
\end{split}
\end{equation*}
with suitable constants $N_\epsilon(\|\nabla w_1\|_p,\|\nabla w_2\|_q,\|w_i\|_\infty)>0$, $i=1,2$.
\end{lemma}
\begin{proof}
Lemma \ref{auxsol} ensures that $ u \in \mathcal{D}^{1,p}_0(\R^N) \cap C^1_{{\rm loc}}(\R^N)$, while ${\rm H_1}$ entails $f_{w,\epsilon}\in L^1(\R^N)\cap L^{\zeta_1}(\R^N)$. By Lemma \ref{Mingione}, besides ${\rm H_1}$ again, we thus have
\begin{equation*}
\begin{split}
\|\nabla u\|_\infty^{p-1} &\leq c(\|f_{w,\epsilon}\|_{\zeta_1} + \|\nabla u\|_p^{p-1}) \\
&\leq c[\max\{1,\epsilon^{\alpha_1}\} \|a_1\|_{\zeta_1} (\|w_2\|_\infty^{\beta_1} + \|\nabla w_1\|_\infty^{\gamma_1} + \|\nabla w_2\|_\infty^{\delta_1}) + \|\nabla u\|_p^{p-1}]  \\
&\leq c_\epsilon (\|w_2\|_\infty^{\beta_1} + \|\nabla w_1\|_\infty^{\gamma_1} + \|\nabla w_2\|_\infty^{\delta_1} + \|\nabla u\|_p^{p-1}).
\end{split}
\end{equation*}
Now, using Lemma \ref{auxunifest} yields
\begin{equation*}
\begin{split}
\|\nabla u\|_\infty^{p-1} &\leq c_\epsilon (\|w_2\|_\infty^{\beta_1} + \|\nabla w_1\|_\infty^{\gamma_1} + \|\nabla w_2\|_\infty^{\delta_1} + \|\nabla w_1\|_p^{\gamma_1} + \|\nabla w_2\|_q^{\eta_1} + 1)\\
&\leq N_\epsilon(\|\nabla w_1\|_p,\|\nabla w_2\|_q,\|w_2\|_\infty) (1 + \|\nabla w_1\|_\infty^{\gamma_1} + \|\nabla w_2\|_\infty^{\delta_1}),
\end{split}
\end{equation*}
where
$$N_\epsilon(\|\nabla w_1\|_p,\|\nabla w_2\|_q,\|w_2\|_\infty):=c_\epsilon(1 + \|w_2\|_\infty^{\beta_1} + \|\nabla w_1\|_p^{\gamma_1} + \|\nabla w_2\|_q^{\eta_1}).$$
This shows the first inequality. The other is analogous.
\end{proof}
\subsection{Regularizing the right-hand side}
Let $ {\rm H_1} $ be satisfied. Given $\epsilon>0$, define 
\begin{equation*}
\begin{split}
\mathcal{R}_\epsilon:=\{(w_1,w_2)\in \mathcal{C}^1_+:&\|\nabla w_1\|_p \leq A_1,\,\|\nabla w_2\|_q \leq A_2,\\
& \|w_i\|_\infty \leq B_i, \, \|\nabla w_i\|_\infty \leq C_i,\, i=1,2\},
\end{split}
\end{equation*}
with $ A_i,B_i,C_i > 0 $, $ i=1,2 $, such that
\begin{equation}
\label{conds}
\left\{
\begin{array}{ll}
A_1^{p-1} &\geq L_\epsilon (1 + A_1^{\gamma_1} + A_2^{\eta_1}), \\
A_2^{q-1} &\geq L_\epsilon (1 + A_1^{\eta_2} + A_2^{\delta_2}), \\
B_1,B_2 &\geq M_\epsilon(A_1,A_2), \\
C_1^{p-1} &\geq N_\epsilon(A_1,A_2,B_2) (1 + C_1^{\gamma_1} + C_2^{\delta_1}), \\
C_2^{q-1} &\geq N_\epsilon(A_1,A_2,B_1) (1 + C_1^{\gamma_2} + C_2^{\delta_2}),
\end{array}
\right.
\end{equation}
and $ L_\epsilon $, $ M_\epsilon(\cdot,\cdot) $, $ N_\epsilon(\cdot,\cdot,\cdot) $ stemming from Lemmas \ref{auxunifest}--\ref{auxgradest}. Apropos, system \eqref{conds} admits solutions. In fact, by $ {\rm H_1} $, we can pick
\begin{equation}\label{sigma}
1 < \sigma < \frac{(p-1)(q-1)}{\eta_1 \eta_2}.
\end{equation}
If $ A_1 := K^{\frac{1}{\eta_2}} $ and $ A_2 := K^{\frac{\sigma}{q-1}} $ then the first two inequalities of \eqref{conds} become
\begin{equation*}
K^{\frac{p-1}{\eta_2}}\geq L_\epsilon (1 + K^{\frac{\gamma_1}{\eta_2}}  + K^{\frac{\sigma \eta_1}{q-1}}),\quad
K^{\sigma}\geq L_\epsilon (1 + K + K^{\frac{\sigma \delta_2}{q-1}}),
\end{equation*}
which, due to \eqref{sigma}, are true for any sufficiently large $K>0$. Next, choose
\begin{equation*}
B_1 := B_2 := M_\epsilon(K^{\frac{1}{\eta_2}},K^{\frac{\sigma}{q-1}}).
\end{equation*}
With $ A_i $, $ B_i $ as above, set $ C_1 := H^{\frac{1}{\eta_2}} $ and $ C_2 := H^{\frac{\sigma}{q-1}} $. The last two inequalities in \eqref{conds} rewrite
\begin{equation*}
\begin{split}
H^{\frac{p-1}{\eta_2}} &\geq N_\epsilon(A_1,A_2,B_2) (1 + H^{\frac{\gamma_1}{\eta_2}}  + H^{\frac{\sigma \delta_1}{q-1}}), \\
H^{\sigma} &\geq N_\epsilon(A_1,A_2,B_1) (1 + H^{\frac{\gamma_2}{\eta_2}} + H^{\frac{\sigma \delta_2}{q-1}}).
\end{split}
\end{equation*}
Thanks to \eqref{sigma} again, they hold for every $ H>0 $ big enough.

On the \textit{trapping region} $\mathcal{R}_\epsilon$ we will consider the topology induced by that of $ X $. Let us now investigate the regularized problem
\begin{equation}\label{aux2}
\tag{$ {\rm P}^\epsilon $}
\left\{
\begin{alignedat}{2}
-\Delta_p u &= f(x,u+\epsilon,v,\nabla u, \nabla v) \; \; &&\mbox{in} \; \; \R^N, \\
-\Delta_q v &= g(x,u,v+\epsilon,\nabla u, \nabla v) \; \; &&\mbox{in} \; \; \R^N, \\
u,v &> 0 \; \; &&\mbox{in} \; \; \R^N,
\end{alignedat}
\right.
\end{equation}
where $\epsilon\geq 0$. Evidently, $({\rm P}^\epsilon)$ reduces to \eqref{prob} once $\epsilon=0$.
\begin{lemma}\label{aux2sol}
Under ${\rm H_1}$, for every $ \epsilon > 0$ problem \eqref{aux2} possesses a solution $ (u_\epsilon,v_\epsilon) \in \mathcal{C}^{1,\alpha}_+ $.
\end{lemma}
\begin{proof}
Fix $\epsilon>0$ and define, provided $w\in\mathcal{R}_\epsilon$,
\begin{equation*}
T_\epsilon(w):= (u,v),\;\;\mbox{with $(u,v)$ being the unique solution to \eqref{aux}};
\end{equation*}
cf. Lemma \ref{auxsol}. From Lemmas \ref{auxunifest}--\ref{auxgradest}, besides \eqref{conds}, it follows $T_\epsilon (\mathcal{R}_\epsilon)\subseteq \mathcal{R}_\epsilon$. 
\vskip2pt
\noindent\emph{Claim 1.} $T_\epsilon (\mathcal{R}_\epsilon)$ is relatively compact in $X$.
\vskip2pt
\noindent To see this, pick $\{w_n\}\subseteq\mathcal{R}_\epsilon $, put
$$w_n:= (w_{1,n},w_{2,n}),\quad (u_n,v_n):= T_\epsilon(w_n),\quad n\in\N,$$
and understand any convergence up to sub-sequences. Since $\{T_\epsilon (w_n)\}\subseteq\mathcal{R}_\epsilon$ while $X$ is reflexive, $\{(u_n,v_n)\}$ weakly converges to a point $(u,v)\in X$. If $Y_\rho:=L^p(B_\rho) \times L^q(B_\rho)$, $\rho>0$, then by \eqref{embedding} one has
\begin{equation*}
X \hookrightarrow W^{1,p}(B_\rho) \times W^{1,q}(B_\rho)\hookrightarrow Y_\rho
\end{equation*}
and the embedding $ X \hookrightarrow Y_\rho$ is compact, due to Rellich-Kondrakov's theorem \cite[Theorem 9.16]{B}. Thus, $ (u_n,v_n) \to (u,v) $ in $ Y_\rho$. Let us next verify that
\begin{equation}\label{Cantor}
(u_n(x),v_n(x)) \to (u(x),v(x)) \;\;\mbox{for almost every} \; \; x\in\R^N.
\end{equation}
In fact, $(u_n,v_n) \to (u,v)$ in $ Y_1 $ yields a sub-sequence $ \{(u_n^{(1)},v_n^{(1)})\}$ of $ \{(u_n,v_n)\}$ such that
$$ (u_n^{(1)}(x),v_n^{(1)}(x)) \to (u(x),v(x))\;\;\mbox{for almost all}\;\; x\in B_1\, .$$
Since $ (u_n^{(1)},v_n^{(1)}) \to (u,v) $ in $ Y_2 $, we can extract a sub-sequence $ \{(u_n^{(2)},v_n^{(2)})\}$ from $ \{(u_n^{(1)},v_n^{(1)})\} $ fulfilling
$$(u_n^{(2)}(x),v_n^{(2)}(x)) \to (u(x),v(x))\;\;\mbox{for almost every}\;\; x\in B_2\, .$$
By induction, to each $ k \geq 2 $ there corresponds a sub-sequence $ \{(u_n^{(k)},v_n^{(k)})\} $ of
$ \{(u_n^{(k-1)},v_n^{(k-1)})\}$ such that 
$$ (u_n^{(k)}(x),v_n^{(k)}(x)) \to (u(x),v(x))\;\;\mbox{for almost all}\;\; x\in B_k\, .$$
Now, Cantor's diagonal procedure leads to $ (u_n^{(n)},v_n^{(n)}) \to (u,v) $ a.e. in $ \R^N $, because $ \bigcup_{k=1}^\infty B_k = \R^N $, and \eqref{Cantor} follows.\\
Through ${\rm H_1(f)}$, besides the inclusion $\{w_n\}\subseteq\mathcal{R}_\epsilon $, we get
\begin{equation}\label{S+}
\begin{split}
&\int_{\R^N} |\nabla u_n|^{p-2} \nabla u_n \nabla (u_n-u) {\rm d}x \\
&= \int_{\R^N} f(\cdot,w_{1,n}+\epsilon,w_{2,n},\nabla w_n) (u_n-u) {\rm d}x  \\
&\leq \int_{\R^N} f(\cdot,w_{1,n}+\epsilon,w_{2,n},\nabla w_n) |u_n-u| {\rm d}x \\
&\leq c_\epsilon\int_{\R^N} a_1|u_n-u| {\rm d}x\quad\forall\, n\in\N,
\end{split}
\end{equation}
with $c_\epsilon:=\hat{m}_1(\epsilon^{\alpha_1}B_2^{\beta_1}+C_1^{\gamma_1}+C_2^{\delta_1})$. Using $T_\epsilon(\mathcal{R}_\epsilon)\subseteq \mathcal{R}_\epsilon$ and \eqref{Cantor} one has
\begin{equation*}
a_1 |u_n-u| \leq 2B_1 a_1 \in L^1(\R^N), \quad n\in\N.
\end{equation*}
So, by \eqref{Cantor}--\eqref{S+}, Lebesgue's Theorem entails
\begin{equation*}
\limsup_{n\to \infty} \int_{\R^N} |\nabla u_n|^{p-2} \nabla u_n \nabla (u_n-u) {\rm d}x 
\leq c_\epsilon\lim_{n \to \infty} \int_{\R^N} a_1|u_n-u| {\rm d}x = 0.
\end{equation*}
Now, recall (cf., e.g., \cite[Proposition 2.2]{MMM}) that the operator $(-\Delta_p,\mathcal{D}^{1,p}_0(\R^N))$ is of type ${\rm (S)_+}$  to achieve $u_n\to u$ in $\mathcal{D}^{1,p}_0(\R^N)$. A similar reasoning applies to $\{v_n\} $. 
\vskip2pt
\textit{Claim 2.} $T_\epsilon:\mathcal{R}_\epsilon\to\mathcal{R}_\epsilon$ is continuous.
\vskip2pt
\noindent Let $\{w_n\}\subseteq\mathcal{R}_\epsilon$ and $w\in\mathcal{R}_\epsilon$ satisfy $w_n\to w$ in $X$. Thanks to \eqref{embedding}, Theorem 4.9 of \cite{B} provides
\begin{equation}\label{pointconvw}
w_n(x) \to w(x) \;\; \mbox{and} \;\; \nabla w_n(x) \to \nabla w(x) \;\; \mbox{for almost every} \; \; x\in\R^N.
\end{equation}
Morever, if $(u_n,v_n):= T_\epsilon(w_n)$, $n\in\N$, then there exists a point $(u,v)\in X$ such that $(u_n,v_n)\to (u,v)$ in $X$; see the proof of Claim 1. Arguing as before, we obtain
\begin{equation}\label{pointconvu}
u_n(x) \to u(x)\;\;\mbox{and}\;\; \nabla u_n(x) \to \nabla u(x) \;\; \mbox{for almost every} \; \; x\in\R^N.
\end{equation}
Since $\Vert\nabla u_n\Vert_p\leq A_1$ whatever $n$, the sequence $\{|\nabla u_n|^{p-2}\nabla u_n\}\subseteq L^{p'}(\R^N)$ turns out bounded. Due to \eqref{pointconvu} and \cite[Exercise 4.16]{B}, this yields
\begin{equation}\label{seqLHS}
\lim_{n \to \infty} \int_{\R^N} |\nabla u_n|^{p-2} \nabla u_n \nabla \phi {\rm d}x = \int_{\R^N} |\nabla u|^{p-2} \nabla u \nabla \phi {\rm d}x,\;\;\phi\in\mathcal{D}^{1,p}_0(\R^N). 
\end{equation}
On the other hand, 
\begin{equation}\label{seqRHS}
\lim_{n \to \infty}\int_{\R^N} f(\cdot,w_{1,n}+\epsilon,w_{2,n},\nabla w_n) \phi\,{\rm d}x
= \int_{\R^N} f(\cdot,w_1+\epsilon,w_2,\nabla w) \phi\,{\rm d}x
\end{equation}
by Lebesgue's Theorem jointly with \eqref{pointconvw} and the inequality
\begin{equation*}
f(\cdot,w_{1,n}+\epsilon,w_{2,n},\nabla w_n) |\phi| \leq c_\epsilon a_1 |\phi| \in L^1(\R^N)\;\;\forall\, n\in\N,
\end{equation*}
which easily arises from ${\rm H_1(f)}$ besides the choice of $\mathcal{R}_\epsilon$. Finally,
\begin{equation}\label{seq}
\int_{\R^N}|\nabla u_n|^{p-2}\nabla u_n \nabla \phi{\rm d}x
=\int_{\R^N} f(\cdot,w_{1,n}+\epsilon,w_{2,n},\nabla w_n) \phi{\rm d}x,\;\; n\in\N,
\end{equation}
because $(u_n,v_n)$ solves $({\rm P}_{w_n}^\epsilon)$. Gathering \eqref{seqLHS}--\eqref{seq} together we have
\begin{equation*}
\int_{\R^N}|\nabla u|^{p-2}\nabla u\nabla\phi\,{\rm d}x 
= \int_{\R^N} f(\cdot,w_1+\epsilon,w_2,\nabla w) \phi\,{\rm d}x\;\;\forall\phi\in\mathcal{D}^{1,p}_0(\R^N).
\end{equation*}
The same is evidently true for $ v $. So, $ (u,v) $ turns out a solution to \eqref{aux}. Uniqueness forces $ (u,v) = T_\epsilon(w) $, whence $ T_\epsilon(w_n)\to T_\epsilon(w)$. 
\vskip2pt
\noindent Now, Theorem 3.2 in \cite[p. 119]{GD} can be applied, and $T_\epsilon$ admits a fixed point $(u_\epsilon,v_\epsilon)\in\mathcal{R}_\epsilon$. By definition of $T_\epsilon$, the pair $(u_\epsilon,v_\epsilon)$ solves problem \eqref{aux2}, while Lemma \ref{auxsol} gives $(u_\epsilon,v_\epsilon)\in\mathcal{C}^{1,\alpha}_+$.
\end{proof}
\begin{lemma}\label{unifest}
If $ {\rm H_1} $--$ {\rm H_2} $ hold then there exists a constant $L> 0$, independent of $\epsilon\geq 0 $, such that $ \|(u,v)\| \leq L $ for every solution $(u,v) \in X_+$ to \eqref{aux2}.
\end{lemma}
\begin{proof}
Pick $\epsilon\geq 0$ and suppose $(u,v)\in X_+$ solves \eqref{aux2}. As already made in showing \eqref{auxuest}, one arrives at
\begin{equation}\label{uest}
\begin{split}
\|\nabla u\|_p^p &= \int_{\R^N} f(\cdot,u+\epsilon,v,\nabla u,\nabla v) u {\rm d}x \\
&\leq \hat{m}_1 \int_{\R^N} a_1[(u+\epsilon)^{\alpha_1} v^{\beta_1} + |\nabla u|^{\gamma_1} + |\nabla v|^{\delta_1}] u {\rm d}x \\
&\leq \hat{m}_1 \int_{\R^N} a_1 (u^{\alpha_1+1} v^{\beta_1} + |\nabla u|^{\gamma_1}u + |\nabla v|^{\delta_1}u) {\rm d}x \\
&\leq c\left(\|u\|_{p^*}^{\alpha_1+1} \|v\|_{q^*}^{\beta_1} + \|\nabla u\|_p^{\gamma_1} \|u\|_{p^*} +
\|\nabla v\|_q^{\delta_1} \|u\|_{p^*}\right) \\
&\leq c\left(\|\nabla u\|_p^{\alpha_1+1} \|\nabla v\|_q^{\beta_1} + \|\nabla u\|_p^{\gamma_1+1} +
\|\nabla v\|_q^{\delta_1} \|\nabla u\|_p\right) \\
&\leq c \max\{1,\|\nabla u\|_p^{\gamma_1+1}\} \max\{1,\|\nabla v\|_q^{\eta_1}\}.
\end{split}
\end{equation}
Likewise, 
\begin{equation}\label{vest}
\|\nabla v\|_q^q \leq c \max\{1,\|\nabla v\|_q^{\delta_2+1}\} \max\{1,\|\nabla u\|_p^{\eta_2}\}.
\end{equation}
It should be noted that the constant $c$ does not depend on $ (u,v) $ and $ \epsilon $. If $\|\nabla v\|_q\leq 1$ or $\|\nabla u\|_p\leq 1$  then \eqref{uest}--\eqref{vest} directly lead to the conclusion, because $\gamma_1 +1<p$ and $\delta_2+1<q$; see ${\rm H}_1$. Hence, we may assume $\min\{\|\nabla u\|_p,\|\nabla v\|_q\}>1$. Dividing \eqref{uest}--\eqref{vest} by $ \|\nabla u\|_p^{\gamma_1+1} $ and $ \|\nabla v\|_q^{\delta_2+1} $, respectively, yields
\begin{equation*}
\|\nabla u\|_p^{p-\gamma_1-1} \leq c \|\nabla v\|_q^{\eta_1},\quad
\|\nabla v\|_q^{q-\delta_2-1} \leq c \|\nabla u\|_p^{\eta_2}.
\end{equation*}
This clearly entails
\begin{equation*}
\|\nabla u\|_p^{p-\gamma_1-1} \leq c \|\nabla u\|_p^{\frac{\eta_1 \eta_2}{q-\delta_2-1}}, \quad 
\|\nabla v\|_q^{q-\delta_2-1} \leq c \|\nabla v\|_q^{\frac{\eta_1 \eta_2}{p-\gamma_1-1}}.
\end{equation*}
The conclusion now follows from $ {\rm H_2} $.
\end{proof}
\begin{lemma}\label{moser}
Let ${\rm H_1}$--${\rm H_2}$ be satisfied. Then there exists $M > 0$, independent of $\epsilon\geq 0$, such that
\begin{equation*}
\max\{\|u\|_\infty,\|v\|_\infty\} \leq M
\end{equation*}
for every solution $ (u,v) \in X_+ $ to \eqref{aux2}.  
\end{lemma}
\begin{proof}
It looks like that of Lemma \ref{auxmoser}. With the same notation, fix $\epsilon\geq 0$, suppose $(u,v) \in X_+$ solves \eqref{aux2}, and define $\Omega_1 := \{x \in \R^N: \, u(x) \geq 1\}$. Moreover, given $z\in L^r(\R^N)$, $r>1$, write $\| z\|_r$ in place of $\| z\|_{L^r(\Omega_1)}$ when no confusion can arise.  Exploiting ${\rm H_1(f)}$ one has
\begin{equation*}\label{trunc}
\int_{\Omega_1} |\nabla u|^{p-2}\nabla u \nabla \phi\, {\rm d}x \leq \hat{m}_1 \int_{\Omega_1} a_1 (v^{\beta_1}+|\nabla u|^{\gamma_1}+|\nabla v|^{\delta_1}) \phi\, {\rm d}x 
\end{equation*}
for all $ \phi \in \mathcal{D}^{1,p}_0(\R^N)_+$; cf. \cite[Lemma 3.2]{MMM}. If $\phi := u_K^{kp+1}$, $k \geq 0$, then Fatou's Lemma,
H\"older's inequality combined with ${\rm H_1(a)}$, Sobolev's embedding \eqref{embedding}, and Lemma \ref{unifest} produce
\begin{equation*}
\begin{split}
\frac{kp+1}{(k+1)^p}\,\|u\|_{(k+1)p^*}^{(k+1)p} &\leq \frac{kp+1}{(k+1)^p}\,\liminf_{K \to +\infty}
\|u_K\|_{(k+1)p^*}^{(k+1)p}  \\
&\leq c \int_{\Omega_1} a_1 (v^{\beta_1}+|\nabla u|^{\gamma_1}+|\nabla v|^{\delta_1}) u^{kp+1} {\rm d}x \\
&\leq c (\|\nabla v\|_q^{\beta_1} + \|\nabla u\|_p^{\gamma_1} + \|\nabla v\|_q^{\delta_1}) \|u\|_{(kp+1)\xi_1}^{kp+1} \\
&\leq c \|u\|_{(kp+1)\xi_1}^{kp+1},
\end{split}
\end{equation*} 
where $\xi_1 \in [1,\frac{p^*}{p})$ fulfills \eqref{xione} while $c$ does not depend on $(u,v)$ and $\epsilon$. We now proceed exactly as in the proof of Lemma \ref{auxmoser}, getting $\Vert u\Vert_\infty\leq M$. The other inequality is analogous.
\end{proof}
\begin{lemma}\label{boundbelow}
Assume ${\rm H_1}$--${\rm H_2}$. Then to every $\rho >0$ there corresponds $\sigma_\rho > 0$ such that
\begin{equation}\label{inf}
\min\left\{\essinf_{B_\rho} u, \essinf_{B_\rho} v\right\} \geq \sigma_\rho
\end{equation}
for all $ (u,v) \in X_+ $ distributional solution of \eqref{aux2}, with $0\leq\epsilon\leq 1$.
\end{lemma}
\begin{proof}
Fix $\rho> 0$. Conditions ${\rm H_1(f)}$--${\rm H_1(g)}$, besides Lemma \ref{moser}, entail 
\begin{equation*}
\begin{split}
f(\cdot,u+\epsilon,v,\nabla u, \nabla v) &\geq m_1\left(\essinf_{B_\rho} a_1\right)(M+1)^{\alpha_1} v^{\beta_1}, \\
g(\cdot,u,v+\epsilon,\nabla u, \nabla v) &\geq m_2\left(\essinf_{B_\rho} a_2\right)(M+1)^{\beta_2} u^{\alpha_2}
\end{split}
\end{equation*}
a.e. in $B_\rho$. From \cite[Theorem 3.1]{DM} it thus follows
\begin{equation*}
\begin{split}
\left(\essinf_{B_\rho} u \right)^{p-1} &\geq \frac{c_\rho}{|B_\rho|} \int_{B_\rho} v^{\beta_1} {\rm d}x \geq c_\rho \left(\essinf_{B_\rho} v \right)^{\beta_1}, \\
\left(\essinf_{B_\rho} v \right)^{q-1} &\geq \frac{c_\rho}{|B_\rho|} \int_{B_\rho} u^{\alpha_2} {\rm d}x \geq c_\rho \left(\essinf_{B_\rho} u \right)^{\alpha_2},
\end{split}
\end{equation*}
which easily give
\begin{equation*}
\essinf_{B_\rho} u \leq c_\rho\left(\essinf_{B_\rho} u \right)^{\frac{(p-1)(q-1)}{\alpha_2 \beta_1}},\quad
\essinf_{B_\rho} v \leq c_\rho\left(\essinf_{B_\rho} v \right)^{\frac{(p-1)(q-1)}{\alpha_2 \beta_1}}.
\end{equation*}
Now, \eqref{inf} is a simple consequence of ${\rm H_1}$, because $\alpha_2\beta_1<(p-1)(q-1)$. 
\end{proof}
\section{Proof of the main result}
\begin{lemma}\label{distrsol}
Under ${\rm H_1}$--${\rm H_3}$, problem \eqref{prob} possesses a distributional solution $(u,v) \in X_+$.
\end{lemma}
\begin{proof}
Let $\epsilon_n:=\frac{1}{n}$, $n \in \N $. Lemma \ref{aux2sol} furnishes a sequence $\{(u_n,v_n)\}\subseteq \mathcal{C}^1_+$ such that $(u_n,v_n)$ solves $({\rm P}^{\epsilon_n})$ for all $n\in\N$. Since $X$ is reflexive, by Lemma \ref{unifest} one has $(u_n,v_n)\rightharpoonup (u,v)$ in $X$, where a sub-sequence is considered when necessary. As before (cf. the proof of Lemma \ref{aux2sol}), this forces \eqref{Cantor}. Moreover, $(u,v) \in X_+ $ because, thanks to Lemma \ref{boundbelow}, to each $\rho> 0 $ there corresponds $\sigma_\rho > 0$ satisfying
\begin{equation}\label{infn}
\min\left\{\inf_{B_\rho} u_n,\;\;\inf_{B_\rho} v_n\right\}\geq \sigma_\rho\;\;\forall\, n\in\N.
\end{equation}
\noindent\emph{Claim.} For every $\rho>0$, and along a sub-sequence if necessary, one has
\begin{equation}\label{cmgrad}
(u_n,v_n) \to (u, v) \;\; \mbox{in}\;\; W^{1,p}(B_\rho)\times W^{1,q}(B_\rho),
\end{equation}
Likewise the proof of \eqref{pointconvw}, this will force
\begin{equation}\label{cpgrad}
(\nabla u_n,\nabla v_n)\to (\nabla u,\nabla v)\;\;\mbox{a.e. in}\;\; \R^N.
\end{equation}
Let $\rho>0$. Hypothesis $ {\rm H_1} $, \eqref{infn}, Lemma \ref{moser}, and $ {\rm H_3} $ yield
\begin{equation}\label{estonball}
\begin{split}
&f(\cdot,u_n+1/n,v_n,\nabla u_n,\nabla v_n) \\
&\leq \hat{m}_1 a_1\left[(u_n+1/n)^{\alpha_1}v_n^{\beta_1}+|\nabla u_n|^{\gamma_1}
+|\nabla v_n|^{\delta_1}\right] \\
&\leq \hat{m}_1\left(\sigma_{2\rho}^{\alpha_1}M^{\beta_1}+|\nabla u_n|^{\gamma_1}+|\nabla v_n|^{\delta_1}\right) a_1 \in L^2(B_{2\rho})
\end{split}
\quad\mbox{ in }B_{2\rho}
\end{equation}
whatever $n$. So, \cite[Theorem 2.1]{CM} combined with Lemma \ref{unifest} ensure that $\{|\nabla u_n|^{p-2}\nabla u_n\}$ turns out bounded in $W^{1,2}(B_\rho)$. Since $p>2-\frac{1}{N}$, by Rellich-Kondrakov's theorem \cite[Theorem 9.16]{B}, the embedding $W^{1,2}(B_\rho)\hookrightarrow L^{p'}(B_\rho)$ is compact. Thus, up to sub-sequences,
\begin{equation}\label{strongconv}
|\nabla u_n|^{p-2} \nabla u_n \to U\;\;\mbox{in} \; \; L^{p'}(B_\rho).
\end{equation}
Next, observe that the linear operator
\begin{equation*}
z\in\mathcal{D}^{1,p}_0(\R^N)\mapsto\nabla z\lfloor_{B\rho}\in L^p(B_\rho)
\end{equation*}
turns out well-defined and strongly continuous. Therefore,
\begin{equation}\label{weakconv}
\nabla u_n \rightharpoonup \nabla u\;\;\mbox{in} \; \; L^p(B_\rho);
\end{equation}
cf. \cite[Theorem 3.10]{B}. Gathering \cite[Proposition 3.5]{B} and \eqref{strongconv}--\eqref{weakconv} together gives
\begin{equation*}
\lim_{n \to \infty} \int_{B_\rho} |\nabla u_n|^{p-2} \nabla u_n \nabla (u_n-u) {\rm d}x = 0.
\end{equation*}
Since $(-\Delta_p,W^{1,p}(B_\rho))$ enjoys the ${\rm (S)_+}$-property, we easily achieve $u_n \to u$ in $W^{1,p}(B_\rho)$. A similar conclusion holds for $ \{v_n\} $, which shows \eqref{cmgrad}. 
\vskip2pt
Now, to verify that $(u,v)$ is a distributional solution of \eqref{prob}, pick any $(\phi_1,\phi_2)\in C^\infty_0(\R^N)^2$ and choose $\rho>0$ fulfilling
$$\supp\phi_1\cup\supp\phi_2\subseteq B_\rho.$$
By \eqref{cmgrad}, \cite[Theorem 4.9]{B} furnishes $(h,k)\in L^p(B_\rho)\times L^q(B_\rho)$ such that
\begin{equation*}\label{convLeb}
|\nabla u_n| \leq h, \quad |\nabla v_n| \leq k\quad\mbox{a.e. in $B_\rho$ and for all $n\in\N$,}
\end{equation*}
whence
\begin{equation*}
f(\cdot,u_n+1/n,v_n,\nabla u_n,\nabla v_n)|\phi_1| 
\leq c_\rho (1+h^{\gamma_1}+k^{\delta_1})a_1|\phi_1|\in L^1(\R^N),\; n\in\N,
\end{equation*}
through \eqref{estonball}. So, thanks to \eqref{Cantor} and \eqref{cpgrad}, Lebesgue's Theorem entails
\begin{equation*}
\lim_{n \to \infty}\int_{\R^N} f(\cdot,u_n+1/n,v_n,\nabla u_n,\nabla v_n)\phi_1{\rm d}x
=\int_{\R^N} f(\cdot,u,v,\nabla u,\nabla v)\phi_1 {\rm d}x.
\end{equation*}
On account of \eqref{strongconv} and \eqref{cpgrad}, we then get
\begin{equation*}
\lim_{n \to \infty}\int_{\R^N} |\nabla u_n|^{p-2} \nabla u_n \nabla\phi_1{\rm d}x
=\int_{\R^N} |\nabla u|^{p-2} \nabla u \nabla\phi_1{\rm d}x.
\end{equation*}
Recalling that each $(u_n,v_n)$ weakly solves $({\rm P}^{\epsilon_n})$ produces
$$\int_{\R^N} |\nabla u|^{p-2} \nabla u \nabla\phi_1{\rm d}x
=\int_{\R^N} f(\cdot,u,v,\nabla u,\nabla v)\phi_1 {\rm d}x.$$
Likewise, 
$$\int_{\R^N} |\nabla v|^{p-2} \nabla v \nabla\phi_2{\rm d}x
=\int_{\R^N} g(\cdot,u,v,\nabla u,\nabla v)\phi_2 {\rm d}x,$$
and the assertion follows.
\end{proof}
\begin{lemma}\label{weaksol}
Let ${\rm H_1}$--${\rm H_2}$ be satisfied and let $(u,v) \in X_+$ be a distributional solution to problem \eqref{prob}. Then $(u,v)$ weakly solves \eqref{prob}.
\end{lemma}
\begin{proof}
We evidently have, for any $\phi\in\mathcal{D}^{1,p}_0(\R^N)$,
\begin{equation}\label{splitting}
\phi= \phi^+- \phi^-.
\end{equation}
Due to the nature of $\phi^+$, a localization-regularization procedure will be necessary. With this aim, fix $\theta \in C^\infty([0,+\infty))$ such that
\begin{equation}\label{theta}
\theta(t)=\left\{
\begin{array}{ll}
1\;\; &\mbox{if} \; \; 0 \leq t \leq 1, \\
0\;\; &\mbox{when} \; \; t \geq 2,
\end{array}
\right. \quad \theta\;\;\mbox{is decreasing in}\;(1,2)
\end{equation}
and a sequence $\{\rho_k\}\subseteq C^\infty_0(\R^N)$ of standard mollifiers \cite[p. 108]{B}. Define, for every $n,k\in\N$, 
$$\theta_n(\cdot) := \theta(|\cdot|/n)\in C^\infty_0(\R^N),\quad\phi_n := \theta_n\,\phi^+ \in\mathcal{D}^{1,p}_0(\R^N),$$
$$\psi_{k,n} := \rho_k * \phi_n \in C^\infty_0(\R^N).$$
Using \eqref{theta} we easily get $\phi_n \uparrow \phi^+$. Moreover, $\displaystyle{\lim_{k\to\infty}}\psi_{k,n}= \phi_n$ in $\mathcal{D}^{1,p}_0(\R^N)$, which entails
\begin{equation}\label{convkleft}
\lim_{k \to \infty}\int_{\R^N} |\nabla u|^{p-2}\nabla u\nabla \psi_{k,n} {\rm d}x =
\int_{\R^N} |\nabla u|^{p-2}\nabla u \nabla \phi_n {\rm d}x,\quad n\in\N.
\end{equation}
If, to shorten notation, $\hat{f}:= f(\cdot,u,v,\nabla u, \nabla v)$ then the linear functional
$$\psi\in\mathcal{D}^{1,p}_0(\R^N)\mapsto\int_{B_{2n+2}}\hat{f} \psi\,{\rm d}x$$
turns out continuous. In fact, Lemmas \ref{moser}--\ref{boundbelow},  H\"older's inequality combined with ${\rm H_1(a)}$, and \eqref{embedding} produce
\begin{equation*}
\int_{B_{2n+2}} a_1 u^{\alpha_1} v^{\beta_1} |\psi| {\rm d}x \\
\leq\sigma_{2n+2}^{\alpha_1} M^{\beta_1} \|a_1\|_{(p^*)'} \|\psi\|_{p^*} \leq c_n \|\nabla \psi\|_p.
\end{equation*}
Now, the assertion follows from ${\rm H_1(f)}$, because convection terms can be estimated as already made in \eqref{uest}.\\
Observe next that
\begin{equation*}\label{support}
\supp\psi_{k,n}\subseteq \overline{\supp \rho_k +\supp\phi_n} \subseteq\overline{B_1 + B_{2n}} \subseteq B_{2n+2}\quad\forall\, n,k\in\N;
\end{equation*}
see \cite[Proposition 4.18]{B}. Hence,
\begin{equation}\label{convkright}
\begin{split}
\lim_{k \to \infty}\int_{\R^N} \hat{f} \psi_{k,n} {\rm d}x &=\lim_{k \to \infty} \int_{B_{2n+2}} \hat{f} \psi_{k,n} {\rm d}x\\
&=\int_{B_{2n+2}} \hat{f} \phi_n {\rm d}x = \int_{\R^N} \hat{f} \phi_n {\rm d}x.
\end{split}
\end{equation}
On the other hand, the hypothesis $(u,v) \in X_+$ distributional solution to \eqref{prob} evidently forces
\begin{equation*}\label{idkn}
\int_{\R^N}|\nabla u|^{p-2}\nabla u\nabla \psi_{k,n}{\rm d}x =\int_{\R^N}\hat{f}\psi_{k,n}{\rm d}x,\quad k,n \in \N.
\end{equation*}
Letting $k \to+\infty$ and exploiting \eqref{convkleft}--\eqref{convkright} we thus achieve
\begin{equation}\label{idn}
\int_{\R^N}|\nabla u|^{p-2}\nabla u\nabla\phi_n {\rm d}x = \int_{\R^N} \hat{f}\phi_n {\rm d}x\quad\forall\, n\in\N.
\end{equation}
\emph{Claim.} $\phi_n\to\phi^+$ in $\mathcal{D}^{1,p}_0(\R^N)$.
\vskip2pt
\noindent In fact, for every $n\in\N$ one has
\begin{equation}\label{gradconv}
\begin{split}
\int_{\R^N} & |\nabla \phi_n - \nabla \phi^+|^p {\rm d}x = \int_{\R^N} |\phi^+ \nabla \theta_n + \theta_n \nabla \phi^+ - \nabla \phi^+|^p {\rm d}x \\
&\leq c\left( \int_{\R^N} (1-\theta_n)^p |\nabla \phi^+|^p {\rm d}x + \int_{B_{2n} \setminus B_n} |\nabla \theta_n|^p (\phi^+)^p {\rm d}x \right) \\
&\leq c \int_{\R^N} (1-\theta_n)^p |\nabla \phi^+|^p {\rm d}x \\
&+ c \left( \int_{B_{2n} \setminus B_n} |\nabla \theta_n|^{\frac{pp^*}{p^*-p}} {\rm d}x \right)^{1-\frac{p}{p^*}} \left( \int_{B_{2n} \setminus B_n} (\phi^+)^{p^*} {\rm d}x \right)^{\frac{p}{p^*}} \\
&= c \int_{\R^N} (1-\theta_n)^p |\nabla \phi^+|^p {\rm d}x + c \|\nabla \theta_n\|_N^p \left( \int_{B_{2n} \setminus B_n} (\phi^+)^{p^*} {\rm d}x \right)^{\frac{p}{p^*}}.
\end{split}
\end{equation}
Recall that $\phi^+\in\mathcal{D}^{1,p}_0(\R^N)$. By \eqref{theta}, Lebesgue's Theorem yields
\begin{equation}\label{aux1}
\lim_{n \to \infty} \int_{\R^N} (1-\theta_n)^p |\nabla \phi^+|^p {\rm d}x = 0
\end{equation}
while, on account of \eqref{embedding},
\begin{equation}\label{auxil2}
\lim_{n \to \infty} \int_{B_{2n} \setminus B_n} (\phi^+)^{p^*} {\rm d}x = 0.
\end{equation}
Since, due to \eqref{theta} again,
\begin{equation*}
\int_{\R^N}|\nabla \theta_n|^N {\rm d}x =\frac{1}{n^N}\int_{\R^N} \left| \theta'\left(\frac{|x|}{n}\right) \right|^N{\rm d}x = \int_{\R^N} |\theta'(|x|)|^N {\rm d}x<+\infty\;\;\forall\, n\in\N,
\end{equation*}
gathering \eqref{gradconv}--\eqref{auxil2} together shows the claim.
\vskip2pt
Consequently,
\begin{equation}\label{convnleft}
\lim_{n \to \infty} \int_{\R^N} |\nabla u|^{p-2}\nabla u \nabla \phi_n {\rm d}x = \int_{\R^N} |\nabla u|^{p-2}\nabla u \nabla \phi^+ {\rm d}x.
\end{equation}
From $\phi_n\uparrow \phi^+$ and $\hat{f}\geq 0$ it then follows
\begin{equation}\label{convnright}
\lim_{n \to \infty}\int_{\R^N} \hat{f} \phi_n {\rm d}x = \int_{\R^N} \hat{f} \phi^+ {\rm d}x
\end{equation}
by Beppo Levi's Theorem. Through \eqref{idn}, \eqref{convnleft}--\eqref{convnright} we thus arrive at
\begin{equation*}\label{id1}
\int_{\R^N} |\nabla u|^{p-2}\nabla u \nabla \phi^+ {\rm d}x = \int_{\R^N} \hat{f} \phi^+ {\rm d}x.
\end{equation*}
Likewise, one has
\begin{equation*}\label{id2}
\int_{\R^N} |\nabla u|^{p-2}\nabla u \nabla \phi^- {\rm d}x = \int_{\R^N} \hat{f} \phi^- {\rm d}x,
\end{equation*}
whence (cf. \eqref{splitting})
\begin{equation*}
\int_{\R^N}|\nabla u|^{p-2}\nabla u\nabla \phi\,{\rm d}x =\int_{\R^N}f(\cdot,u,v,\nabla u, \nabla v)\phi\,
{\rm d}x\quad\forall\,\phi\in\mathcal{D}^{1,p}_0(\R^N). 
\end{equation*}
An analogous argument applies to the second equation in \eqref{prob}.
\end{proof}
%
%
\begin{lemma}\label{strongsol}
Let ${\rm H_1}$--${\rm H_3}$ be satisfied and let $(u,v)\in X_+$ be a distributional solution of \eqref{prob}. Then $ (u,v) $ strongly solves \eqref{prob}.
\end{lemma}
\begin{proof}
Reasoning as before (see \eqref{estonball}) provides $f(\cdot,u,v,\nabla u,\nabla v)\in L^2_{\rm loc}(\R^N)$. Thanks to \cite[Theorem 2.1]{CM}, this implies $|\nabla u|^{p-2}\nabla u \in W^{1,2}_{\rm loc}(\R^N)$. Moreover, 
\begin{equation*}
-\Delta_p u(x) = f(x,u(x),v(x),\nabla u(x),\nabla v(x))\quad\mbox{a.e. in }\R^N
\end{equation*}
because of \cite[Corollary 4.24]{B}. Similarly about $v$ and the other equation.
\end{proof}
\noindent\emph{Proof of Theorem \ref{main}.}\\
Lemmas \ref{distrsol}--\ref{strongsol} directly give the conclusion.
\begin{rmk}\label{regularity} If ${\rm H_3} $ is replaced by the stronger condition
\begin{itemize}
\item[\underline{$ {\rm H_3'} $}] One has
\begin{equation*}
\begin{split}
\frac{1}{s_p} + \max \left\{ \frac{\gamma_1}{p},\frac{\delta_1}{q} \right\} < \frac{1}{p'N},\qquad
\frac{1}{s_q} + \max \left\{ \frac{\gamma_2}{p},\frac{\delta_2}{q} \right\} < \frac{1}{q'N}
\end{split}
\end{equation*}
\end{itemize}
then any distributional solution $(u,v) \in X_+ $ to \eqref{prob} actually lies in $\mathcal{C}^{1,\alpha}_+ $. To show this, pick $\hat{s}_p, \hat{s}_q>0$ such that
\begin{equation*}
\frac{1}{s_p} +\max \left\{ \frac{\gamma_1}{p},\frac{\delta_1}{q} \right\} \leq \frac{1}{\hat{s}_p}< \frac{1}{p'N}\, ,\quad
\frac{1}{s_q} + \max \left\{ \frac{\gamma_2}{p},\frac{\delta_2}{q} \right\}\leq \frac{1}{\hat{s}_q} < \frac{1}{q'N}.
\end{equation*}
As in the proof of \eqref{estonball}, for every $\rho > 0$ one has
\begin{equation*}
\begin{split}
f(\cdot,u,v,\nabla u,\nabla v) &\leq c_\rho a_1 (1 + |\nabla u|^{\gamma_1} + |\nabla v|^{\delta_1})\in L^{\hat{s}_p}(B_\rho), \\
g(\cdot,u,v,\nabla u,\nabla v) &\leq c_\rho a_2 (1 + |\nabla u|^{\gamma_2} + |\nabla v|^{\delta_2})\in L^{\hat{s}_q}(B_\rho).
\end{split}
\end{equation*}
Hence, known nonlinear regularity results \cite[p. 830]{DB} entail $(u,v)\in \mathcal{C}^{1,\alpha}_+ $.
\end{rmk}
%
%
\begin{rmk}\label{equivalence}
Unfortunately, we were not able to find in the literature a definition of strong solution for elliptic equations driven by \emph{non-linear} operators in divergence form. The one adopted here represents a quite natural extension of the semi-linear case $p=2$, where it is asked that the solution $u\in W^{2,2}_{\rm loc}(\R^N)$ and satisfies the differential equation a.e. in $\R^N$; cf. \cite[p. 219]{GT} and \cite[pp. 7--8]{SS}. We cannot expect $u\in W^{2,q}_{\rm loc}(\R^N)$ for some $q>1$, as the example of \cite[Remark 2.7]{CM} shows. Nevertheless, if $(u,v)\in {\mathcal C}^1_+$ is a distributional solution to \eqref{prob} then $u,v \in W^{2,2}_{\rm loc}(\R^N)$ once $1< p,q < 3$; see \cite[p. 2]{S}. On the other hand, each strong solution turns out distributional. So, our notion of strong solution should be read as a distributional solution with an extra differentiability property on the fields $ |\nabla u|^{p-2}\nabla u $ and $ |\nabla v|^{q-2} \nabla v $. 
\end{rmk}
\section*{Acknowledgement}
The authors thank S.J.N. Mosconi for helpful and stimulating discussions.

U.Guarnotta and S.A. Marano were supported by the following research projects: 1) PRIN 2017 `Nonlinear Differential Problems via Variational, Topological and Set-valued Methods' (Grant No. 2017AYM8XW) of MIUR;
2) PRA 2020--2022 Linea 2 `MO.S.A.I.C.' of the University of Catania. 

A. Moussaoui was supported by the Directorate-General of Scientific Research and Technological Development (DGRSDT).

\end{document}